\documentclass[12pt,showkeys]{amsart}

\def\makeheadbox{{%
\hbox to0pt{\vbox{\baselineskip=10dd\hrule\hbox
to\hsize{\vrule\kern3pt\vbox{\kern3pt
\hbox{\bfseries Draft for discussion }
\hbox{Date of this version: 13.04.21}
\kern3pt}\hfil\kern3pt\vrule}\hrule}%
\hss}}}

\usepackage[utf8]{inputenc}
\usepackage{color}
\usepackage{graphicx}
\usepackage{subfigure}
\usepackage{amssymb}
\usepackage{amsmath}
\usepackage{accents}
\usepackage{multirow}
\usepackage{enumerate}
\usepackage{epstopdf}
\usepackage{cite,stmaryrd,txfonts}
\usepackage{tikz}
\usepackage{undertilde}
\usetikzlibrary{snakes}
 \usepackage{float}
\usepackage{makecell}

\usepackage{arydshln}
\usepackage{epic}
\usepackage{empheq}
\usepackage{cases}

\newtheorem{theorem}{Theorem}[section]

\newtheorem{lemma}[theorem]{Lemma}

\newcounter{mnote}
\let\oldmarginpar\marginpar
\renewcommand\marginpar[1]{\-\oldmarginpar[\raggedleft\footnotesize #1]
  {\raggedright\footnotesize #1}}

\usepackage{footnote}
\usepackage{booktabs}

\usepackage{rotating}

\numberwithin{equation}{section}

\usepackage{cite}

\usepackage[colorlinks,linkcolor=black,
anchorcolor=black,
citecolor=black]{hyperref}
\usepackage[shortlabels]{enumitem}
\setlist[enumerate]{nosep}
\usepackage{color, soul}
\soulregister\cite7
\usepackage{setspace}

\def\uu{\undertilde{u}}

\def\uv{\undertilde{v}}

\def\uL{\undertilde{L}}

\def\utau{\undertilde{\tau}}

\def\up{\undertilde{p}}
\def\uq{\undertilde{q}}
\def\uy{\undertilde{y}}
\def\uz{\undertilde{z}}

\def\dv{{\rm div}}

\begin{document}
\title{A mixed element scheme of Helmholtz transmission eigenvalue problem for anisotropic media}


\author{Qing Liu$^{1,2}$, Tiexiang Li$^{1,2,*}$ and Shuo Zhang$^{3,4}$}
\thanks{$^{1}$ School of Mathematics and Shing-Tung Yau Center, Southeast University, Nanjing 210096, China.}
\thanks{ $^{2}$ Nanjing Center for Applied Mathematics, Nanjing 211135,  China.}
\thanks{ $^{3}$ LSEC, ICMSEC, Academy of Mathematics and Systems Science, Chinese Academy of Sciences, Beijing 100190, China.}
\thanks{$^{4}$ University of Chinese Academy of Sciences, Beijing 100049, China.}
\thanks{ * Corresponding author. E-mail address: txli@seu.edu.cn.}


\begin{abstract}
In this paper, we study the Helmholtz transmission eigenvalue problem for inhomogeneous anisotropic media with the index of refraction $n(x)\equiv 1$ in two and three dimension. Starting with a nonlinear fourth order formulation established by Cakoni, Colton and Haddar\cite{Cakoni2009}, by introducing some auxiliary variables, we present an equivalent mixed formulation for this problem, followed up with the finite element discretization. Using the proposed scheme, we rigorously show that the optimal convergence rate for the transmission eigenvalues both on convex and nonconvex domains can be expected. Moreover, by this scheme, we will obtain a sparse generalized eigenvalue problem whose size is so demanding even with a coarse mesh that its smallest few real eigenvalues fail to be solved by the shift and invert method. We partially overcome this critical issue by deflating the almost all of the $\infty$ eigenvalue of huge multiplicity, resulting in a drastic reduction of the matrix size without
deteriorating the sparsity. Extensive numerical examples are reported to demonstrate the effectiveness and efficiency of the proposed scheme.
\end{abstract}

\maketitle


\section{Introduction}

The transmission eigenvalue problem, first introduced by Colton and Monk \cite{COLTON1988} and Kirsch \cite{Kirsch1986}, is generally a system of coupled eigenvalue problems defined on the support of scattering objects. As the transmission eigenvalues carry some important information about the material properties of the scattering object\cite{Giorgi2012,Peters2016,Sun2010}, the problem has found more and more applications in practice and is an important research field of inverse scattering theory, especially in the case of acoustic and electromagnetic waves\cite{Cakoni2014,Dirac1953888,Cakoni2009,Cakoni20077}.

In this paper, we consider the following transmission eigenvalue problem for inhomogeneous anisotropic media. Specifically, we will find $k \in \mathbb{C}\backslash \{0\}$ and non-trivial $w$ and $v$ such that
\begin{equation}{\label{2.1}}
\left\{
\begin{aligned}
& \nabla  \cdot A\left( x \right)\nabla w + {k^2}n(x)w = 0 & {\rm{in}} &\ \Omega,\\
&\Delta v + {k^2}v = 0   &{\rm{in}}& \ \Omega,\\
&w=v  & \ {\rm{on}} &\ \partial \Omega,\\
&\nu  \cdot A\left( x \right)\nabla w = \nu  \cdot \nabla v   & \ {\rm{on}} &\ \partial \Omega,
\end{aligned}
\right.
\end{equation}
where $\Omega \subset {\mathbb{R}^d},\ d = 2,3$ is a bounded polygon or polyhedron with boundary $\partial \Omega$, $\nu$ is the unit outward normal vector of $\partial \Omega$, the matrix index of refraction $A(x)  \in {\mathbb{R}^{d \times d}}$ is symmetric positive definite everywhere, and the index of refraction $n(x)>0$ is bounded. The values of $k$ for which the transmission eigenvalue problem \eqref{2.1} has non-trivial solutions $w$ and $v$ are called transmission eigenvalues. This problem arises from the electromagnetic wave scattering problem in two dimension (2D) \cite{Cakoni2013}
 and the acoustic wave scattering problem in three dimension (3D) \cite{Colton2003}, where the smallest real transmission eigenvalue is essential in the reconstruction of the matrix index of refraction \cite{Cakoni2009}. In this paper, we will focus on the case of $n(x)\equiv 1$ on $\Omega$, which implies that the scatterer has the same permeability in 2D or sound speed in 3D as the external medium.

 There have been a few algorithms for numerical calculation of transmission eigenvalues for anisotropic media. Cakoni, Colton and Haddar \cite{Cakoni2009} provided a method for
determining transmission eigenvalues by far field equation. Ji and Sun \cite{Ji20133} proposed a multi-level finite element method for the transmission eigenvalue problem. An and Shen \cite{Shen2014} developed a spectral method to compute the transmission eigenvalues.  Lechleiter and Peters \cite{Lechleiter2015} calculated  transmission eigenvalues by inside-outside duality technique. Xie and Wu \cite{Xie2017} proposed a multi-level correction method based on finite element to solve the transmission eigenvalue problem. Kleefeld and Pieronek \cite{Kleefeld2018} computed the transmission eigenvalues by fundamental solutions method. Gong et al. \cite{https://doi.org/10.48550/arxiv.2001.05340} transformed this problem into an eigenvalue problem of a holomorphic operator function, then Lagrange finite elements and the spectral projection are used to compute the transmission eigenvalues. Liu and Sun \cite{Liu2021} proposed a Bayesian inversion approach for the transmission eigenvalue problem. We particularly note that, for the case $n(x)>1$ or $n(x)<1$, the existence and uniqueness of the solution to \eqref{2.1} were established by Xie and Wu \cite{Xie2017}, which may serve as a theoretical foundation for corresponding numerical algorithms. Yet, to the best of our knowledge, very few numerical results of 3D Helmholtz transmission eigenvalue problem for anisotropic media have been reported.

 Cakoni, Colton and Haddar\cite{Cakoni2009} established the basic theory of the transmission eigenvalue problem for anisotropic media in different approaches depending on whether $n(x) \equiv 1$ or $n(x) \neq 1$. In particular, for the case of $n(x)\equiv1$, the main technical novelty of \cite{Cakoni2009} was to  introduce $\undertilde{u} :=A\nabla w - \nabla v$ and $\lambda = {k^2}$, and then recast \eqref{2.1} into the transmission eigenvalue problem as a fourth order problem as follows
 \begin{equation}{\label{2.3}}
 ( {\nabla \nabla  \cdot  + {\lambda}A^{-1}} ){( {A^{-1} - I} )^{ - 1}}( {\nabla \nabla  \cdot + {\lambda})\uu}  = 0,
\end{equation}
where $I$ is the identity matrix. It is interesting to note that, in order to solve the transmission
eigenvalue problem for isotropic media, theories and algorithms have been developed in \cite{Cakoni201444,Colton2010,Ji2012,Ji2016,Li2018,Xi2018} and so on. Among these studies, the fourth order problem of the transmission eigenvalue problem was also constructed in \cite{Cakoni201444,Ji2012,Ji2016,Xi2018}, but based on an auxiliary variable $w-v$ rather than $A\nabla w-\nabla v$ in anisotropic case.

In this paper, we study the numerical method for \eqref{2.3}, which is seldomly discussed in literature. Assuming that
$${\kappa_*}: = \mathop {\inf }\limits_{x \in \Omega} \mathop {\inf }\limits_{\xi  \in {\mathbb{R}^d},\left|| \xi  \right|| _{{l^2}}= 1} \left( {\xi^\top A\left( x \right)\xi } \right)\  {\rm and} \ {\kappa^*}: = \mathop {\sup }\limits_{x \in \Omega} \mathop {\sup }\limits_{\xi  \in {\mathbb{R}^d} ,\left|| \xi  \right||_{{l^2}} = 1} \left( {\xi^\top A\left( x \right)\xi } \right) $$
satisfy either $0<{\kappa_*}\leq{\kappa^*}<1$ or $1<{\kappa_*}\leq{\kappa^*}<\infty$, by the spectral theory of compact operator, it holds that \cite{Cakoni2009}
\begin{enumerate}
\item If $ \kappa^{\ast}<1$ or $\kappa_{\ast}>1$,  problem \eqref{2.3} has an infinite countable set of real transmission eigenvalues with $+ \infty$ as the only accumulation point.
\item Moreover, denoting by $\kappa_{1}(\Omega)$ the first Dirichlet eigenvalue of $-\Delta$ on $\Omega$, let $k$ be a real transmission eigenvalue and $\lambda = k^2$, the following result holds.
$$
\frac{\lambda}{{\kappa _1}(\Omega) }\ge\left\{
\begin{array}{ll}
\| A^{-1} \|_2^{-1}, & \kappa^{\ast}<1,
\\
1, & \kappa_{\ast}>1.
\end{array}
\right.
$$
\end{enumerate}
Our key idea here is to rewrite the nonlinear eigenvalue problem \eqref{2.3} to an equivalent linear eigenvalue problem by introducing some auxiliary variables, followed up with the finite element discretization and calculation of the smallest few real transmission eigenvalues. The main ingredient of our approach is, as shown below, due to the huge kernel of the $\dv$ operator, a direct introduction of auxiliary variables will lead to an eigenvalue problem of an ill-defined operator, and to avoid this issue, we use the Helmholtz decomposition to deflate the kernel and work directly on $H^1$ and $L^2$ spaces. The final problem is a linear eigenvalue problem with a compact operator. The discretization scheme is easy to implement both on convex and nonconvex connected domains with firm theoretical support. On the other hand, the goal of computing a few smallest real transmission eigenvalues is hindered by the tremendous size of the sparse generalized eigenvalue problem (GEP) obtained by discretization, in that the $\boldsymbol{LDL}$ factorization cannot fit into the computer memory. We partially resolve this critical issue by deflating the almost all of the $\infty$ eigenvalue of huge multiplicity, resulting in a drastic reduction of the matrix size without
deteriorating the sparsity. Specifically, the ratio of the size of the original GEP to that of the GEP we finally solved is about 3 in 2D and 10 in 3D. Both the computational time and cost are significantly reduced, especially in the 3D case.  In a word, this paper provides a practical method for \eqref{2.3} and carries out 3D numerical investigations of \eqref{2.1}, arguably for the first time.

The remaining of this paper is organized as follows. In Section \ref{sec2}, we present an equivalent linear formulation of the transmission eigenvalue problem for anisotropic media. In Section \ref{sec3}, a mixed element scheme is proposed to discrete the mixed formulation. In Section \ref{sec4}, we deflate the almost all of the huge eigenspace associated with the $\infty$ eigenvalue of the GEP obtained by direct discretization, before employing any eigensolver. Numerical examples are presented in Section \ref{sec5}. Finally, we draw some conclusions and discuss some future works in Section \ref{sec6}.

\paragraph{\bf Notations}

To begin with, let ${\oplus}$ represent the direct sum of space or matrix, and $\otimes $ represent the Kronecker product of matrix. Define spaces
$$\uL^2(\Omega):=(L^{2}(\Omega))^{d}, \ L_{0}^{2}(\Omega):=\{ \tau  \in {L^2}(\Omega ):{(\tau,1)  = 0} \},
$$
$$\tilde H_0^1(\Omega) := H_0^1(\Omega)\cap L_{0}^{2}(\Omega), \ \tilde H^1(\Omega) := H^1(\Omega)\cap L_{0}^{2}(\Omega),$$
and
$$
\mathring{H}_0(\dv,\Omega):=\{\utau\in H_0(\dv,\Omega):\dv\utau=0\},
$$
$$
H^1_0(\dv,\Omega):=\{\utau\in H_0(\dv,\Omega):\dv\utau\in H^1_0(\Omega)\}.
$$

Now we introduce the symbol $\eqslantless$ to denote an order of complex numbers. Let $c_{k} = \rho_{k}e^{i{\theta}_{k}}$, $k =1,2$ be two complex numbers, with $\rho_{k}\geq 0$ and $0\leq \theta_{k}< 2\pi$. Then $c_{1}\eqslantless c_{2}$ if and only if one of the items below holds:\\
1. $\rho_{1} = \rho_{2} = 0;$\\
2. $\rho_{1} < \rho_{2} ;$\\
3. $\rho_{1} = \rho_{2}\neq 0$ and $\theta_{1}\geq \theta_{2}$.\\
It is evident that if $c_{1}\eqslantless c_{2}$ and $c_{2}\eqslantless c_{3}$, then $c_{1}\eqslantless c_{3}$. Coherently, we use the symbol $\eqslantgtr$, whereas $c_{2} \eqslantgtr c_{1}$ if and only if $c_{1}\eqslantless c_{2}$.

\section{Equivalent stable linear eigenvalue problems}\label{sec2}

Let $P = (A^{-1}-I)^{-1}$, as shown in \cite{Cakoni2009}, the variational form of \eqref{2.3} is
 \begin{equation}{\label{2.4}}
 ( {P( {\nabla \nabla  \cdot \undertilde{u} + {\lambda}\undertilde{u}} ),( {\nabla \nabla  \cdot \undertilde{v} + {\lambda}A^{-1}\undertilde{v}} )} ) = 0, \ \forall \uv  \in {H_0^1(\dv,\Omega)}.
\end{equation}

In this section, we rewrite the fourth order problem \eqref{2.4} to equivalent stable linear eigenvalue problems. We discuss this problem in two cases: $\kappa^{\ast}<1$ and $\kappa_{\ast}>1$.

\subsection{ Case\ \uppercase\expandafter{\romannumeral1: $\kappa^{\ast}<1$.}}\label{case1}
For an eigenpair  $(\lambda,\uu)$ of \eqref{2.4}, by introducing
$$
\undertilde{y} = \lambda  \undertilde{u},\ \ \mbox{and}\ \   \undertilde{p}= P\nabla \nabla  \cdot \undertilde{u}+ ( I + P)\undertilde{y},
$$
we have for any $\uv\in H^1_0(\dv,\Omega)$, $\uz\in\uL^2(\Omega)$ and $\uq\in \uL^2(\Omega),$
\begin{equation}
\left\{
\begin{aligned}
& ( {P\nabla \nabla  \cdot \undertilde{u},\nabla \nabla  \cdot \undertilde{v}} ) &+& ( {P\undertilde{y},\nabla \nabla  \cdot \undertilde{v}} ) & &=  \lambda ( {\nabla  \cdot \undertilde{u},\nabla  \cdot \undertilde{v}} ) - \lambda ( {\undertilde{p},\undertilde{v}} )\\
& ( {P\nabla \nabla  \cdot \undertilde{u},\undertilde{z}} ) &+ &( {( {I + P} )\undertilde{y},\undertilde{z}} ) & - ( {\undertilde{p},\undertilde{z}} ) &= 0\\
&&-& ( {\undertilde{y},\undertilde{q}} ) & &=  - \lambda  ( {\undertilde{u},\undertilde{q}} ).
\end{aligned}
\right.
\end{equation}
Now, denote
\begin{equation}
(\nabla {\tilde H^1(\Omega)})^\perp:=\{\uz\in \uL^2(\Omega):(\nabla w,\uz)=0,\ \forall\,w\in \tilde{H}^1(\Omega)\}.
\end{equation}
Then $(\nabla {\tilde H^1(\Omega)})^\perp\subset \mathring{H}_0(\dv,\Omega)$. As $\nabla {\tilde H^1(\Omega)}$ is closed in $\uL^2(\Omega)$,
\begin{equation}\label{eq:hd}
\uL^2(\Omega) = \nabla {\tilde H^1(\Omega)}{ \oplus ^ {{\bot}} }(\nabla {\tilde H^1(\Omega)})^\perp.
\end{equation}
Then, we have $\up=\nabla r +(\nabla r)^\perp$ with $r\in\tilde{H}^1(\Omega)$ and $(\nabla r)^\perp\in (\nabla {\tilde H^1(\Omega)})^\perp$, and any $\uq$ can be rewritten as $\uq=\nabla s+(\nabla s)^\perp$ with $s\in\tilde{H}^1(\Omega)$ and $(\nabla s)^\perp\in (\nabla {\tilde H^1(\Omega)})^\perp$, then
\begin{equation}
\left\{
\begin{aligned}
& ( {P\nabla \nabla  \cdot \undertilde{u},\nabla \nabla  \cdot \undertilde{v}} ) &+& ( {P\undertilde{y},\nabla \nabla  \cdot \undertilde{v}} ) & &=  \lambda ( {\nabla  \cdot \undertilde{u},\nabla  \cdot \undertilde{v}} ) - \lambda ( \nabla r+(\nabla r)^\perp,\undertilde{v} )\\
& ( {P\nabla \nabla  \cdot \undertilde{u},\undertilde{z}} ) &+ &( {( {I + P} )\undertilde{y},\undertilde{z}} ) & - ( \nabla r+(\nabla r)^\perp,\undertilde{z} ) &= 0\\
&&-& ( \undertilde{y},\nabla s+(\nabla s)^\perp) & &=  - \lambda  ( \undertilde{u},\nabla s+(\nabla s)^\perp).
\end{aligned}
\right.
\end{equation}
Choosing $\uv\in \mathring{H}_0(\dv,\Omega)$ arbitrarily leads to that $(\nabla r)^\perp=0$. It follows then that
\begin{equation}\label{3.3}
\left\{
\begin{aligned}
& ( {P\nabla \nabla  \cdot \undertilde{u},\nabla \nabla  \cdot \undertilde{v}} ) &+& ( {P\undertilde{y},\nabla \nabla  \cdot \undertilde{v}} ) & &=  \lambda ( {\nabla  \cdot \undertilde{u},\nabla  \cdot \undertilde{v}} ) + \lambda ( r, \nabla  \cdot \undertilde{v} )\\
& ( {P\nabla \nabla  \cdot \undertilde{u},\undertilde{z}} ) &+ &( {( {I + P} )\undertilde{y},\undertilde{z}} ) & - ( \nabla r,\undertilde{z} ) &= 0\\
&&-& ( \undertilde{y},\nabla s) & &=   \lambda  (  \nabla  \cdot\undertilde{u}, s).
\end{aligned}
\right.
\end{equation}
Introducing formally $\varphi=\nabla\cdot\uu$ and $\psi=\nabla\cdot \uv$ leads us to the linear eigenvalue problem: find $\lambda\in\mathbb{C}$ and $(\varphi,\uy,r)\in V:=\tilde{H}^1_0(\Omega)\times \uL^2(\Omega)\times \tilde{H}^1(\Omega)$, such that, for any $(\psi,\uz,s)\in V$,
\begin{equation}\label{3.4}
\left\{
\begin{aligned}
& ( {P\nabla \varphi ,\nabla \psi } ) &+&  ( {P\undertilde{y},\nabla \psi }) & &=  \lambda ( ( {\varphi ,\psi } )+ ( {r,\psi }) ) \\
&( {P\nabla \varphi ,\undertilde{z}} ) &+ &( {( {I + P} )\undertilde{y},\undertilde{z}} ) & - ( {\nabla r,\undertilde{z}} ) &= 0\\
&&-&( {\undertilde{y},\nabla s} )& &=  \lambda ( {\varphi ,s} ).
\end{aligned}
\right.
\end{equation}

Equip $V$ with the norm
$${\| {( {\varphi,\uy,r})} \|_V}= {( {\| \varphi \|_{1,\Omega}^2 + \| \uy \|_{0,\Omega}^2 + \| r \|_{1,\Omega}^2} )^{1/2}},$$
then $V$ is a Hilbert space. Four bilinear forms are defined:
\begin{equation*}
\begin{aligned}
a ( {( {\varphi ,\uy} ),( {\psi ,\uz} )} ): = &( {P\nabla \varphi ,\nabla \psi } ) + ( {P\uy,\nabla \psi } ) + ( {P\nabla \varphi ,\uz} ) + ( {( {I + P} )\uy,\uz} ),\\
b ( {( {\varphi,\uy} ), {s} } ):= &- ( {\uy,\nabla s} ),
\end{aligned}
\end{equation*}
and
\begin{equation*}
\begin{aligned}
{a_V}( {( {\varphi ,\undertilde{y},r} ),( {\psi ,\undertilde{z},s} )} ): =& ( {P\nabla \varphi ,\nabla \psi } ) + ( {P\undertilde{y},\nabla \psi } ) + ( {P\nabla \varphi ,\undertilde{z}} )+ ( {( {I + P} )\undertilde{y},\undertilde{z}} ) - ( {\nabla r,\undertilde{z}} ) - ( {\undertilde{y},\nabla s} ),\\
{b_V}( {( {\varphi ,\undertilde{y},r} ),( {\psi ,\undertilde{z},s} )} ): = &( {\varphi ,\psi } ) + ( {r,\psi }) + ( {\varphi ,s} ).
\end{aligned}
\end{equation*}
Then, $a( { \cdot , \cdot } )$, $b( { \cdot , \cdot } )$, ${a_{V}}( { \cdot , \cdot } )$ and ${b_{V}}( { \cdot , \cdot } )$ are all symmetric, continuous and bounded. Associated with ${a_V}( { \cdot , \cdot })$ and ${b_V}( { \cdot , \cdot } )$,  we define an operator ${T_V}: V\rightarrow V$ by
\begin{equation}\label{eq:deftv}
{a_{V}(T_{V}{( {\varphi ,\undertilde{y},r} ),( {\psi ,\undertilde{z},s})} )}={b_V( {( {\varphi ,\undertilde{y},r} ),( {\psi ,\undertilde{z},s})})}, \ \forall ( {\psi , \undertilde{z},s}) \in V.
\end{equation}

\begin{lemma}\label{lemma 3.2} $T_{V}$ is well-defined and compact on $V$.
\end{lemma}
\begin{proof}
Denote ${\hat{P} ={P + {\kappa ^ * }I}} $, by the Poincar\'e inequality, the following two inequalities hold.
$$
\begin{aligned}
a( {( {\varphi ,\undertilde{y}} ) ,( {\varphi ,\undertilde{y}} )} ) &=\| {{{\hat{P}}^{ - \frac{1}{2}}}P\nabla \varphi  + {{\hat{P}}^{\frac{1}{2}}}\undertilde{y}} \|_{0,\Omega}^2+
  (1 - {\kappa ^ * })\| \undertilde{y} \|_{0,\Omega}^2 + ((P - P{\hat{P}^{ - 1}}P)\nabla \varphi ,\nabla \varphi )\\
 & \geq  (1 - {\kappa ^ * })\| \undertilde{y} \|_{0,\Omega}^2 + \kappa_*(1-\kappa ^ *)/((1-\kappa_*)(2 - \kappa ^ *))\| \nabla \varphi \|_{0, \Omega}^2\\
& \geq  C( {\| \undertilde{y} \|_{0,\Omega}^2 + \| \varphi \|_{1, \Omega}^2} ),
\end{aligned}
$$
and given $s\in\tilde{H}^1(\Omega)$, set $\uy=-\nabla s$, and it follows that
$$
(\uy,-\nabla s)=\|\uy\|_{0,\Omega}\|\nabla s\|_{0,\Omega}\geqslant C\|\uy\|_{0,\Omega}\|s\|_{1,\Omega}.
$$
The inf-sup condition holds as
$$
\mathop {\inf }\limits_{ {s}  \in  {{\tilde H}^1}} \mathop {\sup }\limits_{( {\varphi ,\undertilde{y}} ) \in \tilde H_0^1 \times \uL^2 } \frac{{ b( {( {\varphi ,\undertilde{y}} ),s } )}}{{ {{ {{\| s \|_{1,\Omega}}}}} ( {{\| \undertilde{y} \|_{0,\Omega}} + {\| \varphi  \|_{1,\Omega}} } )}}\geq C>0.
$$
Therefore, given $(\varphi,\uy,r)\in V$, $T_V(\varphi,\uy,r)$ is uniquely defined by \eqref{eq:deftv}, and
$$
\|T_V(\varphi,\uy,r)\|_{V}\leqslant C(\|\varphi\|_{-1,\Omega}+\|r\|_{-1,\Omega}).
$$

Now, let $\{  {\varphi_{i} ,\undertilde{y}_{i},r_{i}}  \}$ be a bounded sequence in $V$, then there is subsequence $\{( \varphi_{{i}_{j}}, \undertilde{y}_{{i}_{j}}, r_{{i}_{j}} )\}$, such that $\{ \varphi_{{i}_{j}}\}$ and $\{ r_{{i}_{j}}\}$ are two Cauchy sequences in $L^{2}(\Omega)$. Therefore, $\{T_{V}(\varphi_{{i}_{j}},  \undertilde{y}_{{i}_{j}}, r_{{i}_{j}} )\}$
is a Cauchy sequence in $V$, which, further, has a limit therein. The proof is completed.
\end{proof}

The lemma below follows by the standard spectral theory of compact operators.

\begin{lemma}\cite{Riesz1916}
The eigenvalues of  $T_{V}$ and \eqref{3.4}, counting multiplicity,  can be listed in a (finite or infinite) sequence as
$$\mu_{1}\eqslantgtr \mu_{2}\eqslantgtr  \cdots  \eqslantgtr 0\ \ \ {\rm{and}}\ \ 0\eqslantless \lambda_{1}\eqslantless \lambda_{2}\eqslantless  \cdots ,$$
 respectively. Moreover, for any $i\in \mathbb{N}^{+}$ such that $\mu _{i}\neq 0$, it holds $\lambda _{i} \mu _{i}=1$.
\end{lemma}

\begin{theorem}\label{thm 3.1}
The eigenvalue problem \eqref{3.4} is equivalent to \eqref{2.4}.
\end{theorem}
\begin{proof}
Let $(\lambda$, $(\varphi,\uy,r))$ be an eigenpair of \eqref{3.4}, then we have
$\uy\in H_0(\dv,\Omega)$, $\varphi=\nabla \cdot(\uy/\lambda)$ and $\nabla r={P\nabla\varphi+(I+P)}\uy. $ Set $\uu=\uy/\lambda$, and it follows that
\begin{equation*}
( {P( {\nabla \nabla  \cdot \uu + {\lambda}\undertilde{u}} ),( {\nabla \nabla  \cdot \undertilde{v} + {\lambda}A^{-1}\undertilde{v}} )} ) = 0, \ \forall \undertilde{v}  \in {H^1_0(\dv,\Omega)}.
\end{equation*}
Namely $(\lambda,\uy/\lambda)$ is an eigenpair of \eqref{2.4}.

On the other hand, if $(\lambda,\uu)$ is an eigenpair of \eqref{2.4}, then choose a unique $r\in\tilde{H}^1(\Omega)$ such that $\nabla r=P\nabla\nabla\cdot\uu+\lambda(I+P)\uu$, $(\nabla\cdot\uu,\lambda\uu,r)\in V$ and $(\lambda$ , $(\nabla\cdot\uu,\lambda\uu,r))$ solves \eqref{3.4}. The proof is completed.
\end{proof}

\subsection{case\ \uppercase\expandafter{\romannumeral2:
$\kappa_{\ast}>1$.}}\label{case2}
Define $Q = (A-I)^{-1}$. For an eigenpair $(\lambda, \uu)$ of \eqref{2.4}, similar to csae \uppercase\expandafter{\romannumeral1},  the following linear eigenvalue problem can be derived.

Find $(\lambda,({ {\varphi ,\undertilde{y},r} }))\in \mathbb{C}\times V$,
such that, for $( {\psi ,\undertilde{z},s} )\in V,$

\begin{equation}\label{3.40}
\left\{
\begin{aligned}
& ( {(I+Q)\nabla \varphi ,\nabla \psi } ) &+&  ( {Q\undertilde{y},\nabla \psi }) & &=  \lambda ( ( {\varphi ,\psi } )+ ( {r,\psi }) ) \\
&( {Q\nabla \varphi ,\undertilde{z}} ) &+ &( {Q\undertilde{y},\undertilde{z}} ) & - ( {\nabla r,\undertilde{z}} ) &= 0\\
&&-&( {\undertilde{y},\nabla s} )& &=  \lambda ( {\varphi ,s} ).
\end{aligned}
\right.
\end{equation}

\noindent The bilinear form is defined on $V$:
\begin{equation*}
\begin{aligned}
{\hat{a}_V}( {( {\varphi ,\undertilde{y},r} ),( {\psi ,\undertilde{z},s} )} ): =& ( {( {I+Q} )\nabla \varphi ,\nabla \psi } ) + ( {{Q}\undertilde{y},\nabla \psi } ) + ( {{Q}\nabla \varphi ,\undertilde{z}} ) + ( {{Q}\undertilde{y},\undertilde{z}} )- ( {\nabla r,\undertilde{z}} ) - ( {\undertilde{y},\nabla s} ),
\end{aligned}
\end{equation*}
then ${\hat{a}_{V}}$ is symmetric, continuous and bounded. Associated with ${\hat{a}_V}\left( { \cdot , \cdot } \right)$ and ${b_V}\left( { \cdot , \cdot } \right)$,  we define an operator ${\hat{T}_V}: V\rightarrow V$ by
\begin{equation*}
{\hat{a}_{V}(\hat{T}_{V}{( {\varphi ,\undertilde{y},r} ),( {\psi ,\undertilde{z},s})} )}={b_V( {( {\varphi ,\undertilde{y},r} ),( {\psi ,\undertilde{z},s})})}, \ \forall ( {\psi , \undertilde{z},s}) \in V.
\end{equation*}
\vspace{-1.2em}
\begin{lemma}\label{lemma 3.3}
$\hat{T}_{V}$ is well-defined and compact on $V$.
\end{lemma}
\begin{proof}
The proof is the same as that of Lemma \ref{lemma 3.2}.
\end{proof}
 \begin{theorem}\label{thm 3.2}
 The eigenvalue problem \eqref{3.40} is equivalent to \eqref{2.4}.
  \end{theorem}
\begin{proof}
The proof is the same as that of Theorem \ref{thm 3.1}.
\end{proof}
\section{ The discretization scheme of transmission eigenvalue problem for anisotropic media}\label{sec3}
For simplicity, only case \uppercase\expandafter{\romannumeral1} is discussed below, since case \uppercase\expandafter{\romannumeral2} is very similar. Let $\{{\mathcal{T}_h}\}$ be a shape regular triangular mesh in 2D or tetrahedral subdivision in 3D of $\Omega$, such that $\bar \Omega = { \cup _{K \in {\mathcal{T}_h}}}\bar K$. Accordingly, the finite element spaces are defined as follows:

\noindent $\bullet$\  \ $\mathcal{L}_{h}$ denotes the linear element space on $\Omega$, and we further define $$\tilde{\mathcal{L}}_{h} := \mathcal{L}_{h}\cap L_0^2(\Omega), \ \mathcal{L}_{h0} := \mathcal{L}_{h}\cap H_0^1(\Omega) \ \ {\rm{and}}\ \ {\tilde{\mathcal{L}}}_{h0} := {\mathcal{L}_{h0}} \cap L_0^2( \Omega ).$$
$\bullet$\ \ $\undertilde{\mathcal{C}}_{h} := (\mathcal{C}_{h})^{d}$ is the vertorial piecewise constant finite element space on $\Omega$.

The discretized mixed eigenvalue problem takes the form: find $\lambda_{h} \in  \mathbb{C}$ and $( {\varphi_{h} ,\uy_{h},r_{h}} )\in V_{h}:=\tilde{\mathcal{L}}_{h0}\times \undertilde{\mathcal{C}}_{h} \times \tilde{\mathcal{L}}_{h}$, such that, for $ ( {\psi_{h} ,\uz_{h},s_{h}} )\in V_{h},$
\begin{equation}\label{3.1}
\left\{
\begin{aligned}
& ( {P\nabla \varphi_{h} ,\nabla \psi_{h} } ) &+&  ( {P\uy_{h},\nabla \psi_{h} } ) & &=  \lambda_{h} ( ( {\varphi_{h} ,\psi_{h} } ) + ( {r_{h},\psi_{h} } )) \\
&( {P\nabla \varphi_{h} ,\uz_{h}} ) &+ &( {( {I + P} )\uy_{h},\uz_{h}} ) & - ( {\nabla r_{h},\uz_{h}} ) &= 0\\
&&-&( {\uy_{h},\nabla s_{h}} )& &=  \lambda_{h}( {\varphi_{h} ,s_{h}} ).
\end{aligned}
\right.
\end{equation}

We propose the following lemma for the well-posedness of the discretized problem \eqref{3.1}.

\begin{lemma}\label{lemma 4.1} There exists a constant $C$, uniformly with respect to $V_h$, such that
\end{lemma}
\begin{equation}\label{4.2}
\mathop {\inf }\limits_{( {\varphi_{h} ,\undertilde{y}_{h},r_{h}} ) \in V_{h}} \mathop {\sup }\limits_{( {\psi_{h} ,\undertilde{z}_{h},s_{h}} ) \in V_{h}} \frac{{a_{V}( {( {\varphi_{h} ,\undertilde{y}_{h},r_{h}} ),( {\psi_{h} ,\undertilde{z}_{h},s_{h}} )} )}}{{{{\| {( \varphi_{h} ,\undertilde{y}_{h},r_{h}} )} \|_{V}}{{\| {( \psi_{h} ,\undertilde{z}_{h},s_{h}} )}\|_{V}}}} \ge C > 0.
\end{equation}
\begin{proof}
Again, by the Poincar\'e inequality, the following two inequalities can be derived.
$$
\begin{aligned}
a( {( {\varphi_{h} ,\undertilde{y}_{h}} ) ,( {\varphi_{h} ,\undertilde{y}_{h}} )} ) &=\| {{{\hat{P}}^{ - \frac{1}{2}}}P\nabla \varphi_{h}  + {{\hat{P}}^{\frac{1}{2}}}\undertilde{y}_{h}} \|_{0,\Omega}^2 + (1 - {\kappa ^ * })\| \undertilde{y}_{h} \|_{0,\Omega}^2 + ((P - P{\hat{P}^{ - 1}}P)\nabla \varphi_{h} ,\nabla \varphi_{h} )\\
 & \geq  (1 - {\kappa ^ * })\| \undertilde{y}_{h} \|_{0,\Omega}^2 + \kappa_*(1-\kappa ^ *)/((1-\kappa_*)(2 - \kappa ^ *))\| \nabla \varphi_{h} \|_{0, \Omega}^2\\
& \geq  C( {\| \undertilde{y}_{h} \|_{0,\Omega}^2 + \| \varphi_{h} \|_{1, \Omega}^2} ),
\end{aligned}
$$
and given $s_{h}\in\tilde{\mathcal{L}}_{h}$, set $\uy_{h}=-\nabla s_{h}$, and it follows that
$$
(\uy_{h},-\nabla s_{h})=\|\uy_{h}\|_{0,\Omega}\|\nabla s_{h}\|_{0,\Omega}\geqslant C\|\uy_{h}\|_{0,\Omega}\|s_{h}\|_{1,\Omega}.
$$
The inf-sup condition holds as
$$
\mathop {\inf }\limits_{ {s}_{h}  \in  \tilde{\mathcal{L}}_{h} } \mathop {\sup }\limits_{( {\varphi_{h} ,\undertilde{y}_{h}} ) \in \tilde{\mathcal{L}}_{h0}\times \undertilde{\mathcal{C}}_{h}} \frac{{ b( {( {\varphi_{h} ,\undertilde{y}_{h}} ),s_{h} } )}}{{ {{ {{\| s_{h} \|_{1,\Omega}}}}} ( {{\| \undertilde{y}_{h} \|_{0,\Omega}} + {\| \varphi_{h}  \|_{1,\Omega}} } )}}\geq C>0.
$$
The proof is completed.
\end{proof}

Associated with $a_{V}(\cdot, \cdot)$ and $b_{V}(\cdot, \cdot)$, we define operator $T_{V_{h}}$, for $ ( {\psi_{h} ,\undertilde{z}_{h},s_{h}} ) \in V_{h}$,
$${a_V}( {{T_{V_{h}}}( {\varphi ,\undertilde{y},r} ),( {\psi_{h} ,\undertilde{z}_{h},s_{h}} )} ) = {b_V}( {( {\varphi ,\undertilde{y},r} ),( {\psi_{h} ,\undertilde{z}_{h},s_{h}} )} ),$$
and operator $S\hspace{-0.07cm}_{V_{h}}$, for $ ( {\psi_{h} ,\undertilde{z}_{h},s_{h}} ) \in V_{h}$,
$${a_V}( {{S\hspace{-0.07cm}_{V_h}}( {\varphi ,\undertilde{y},r} ),( {\psi_{h} ,\undertilde{z}_{h},s_{h}})} ) = {a_V}( {( {\varphi ,\undertilde{y},r} ),( {\psi_{h} ,\undertilde{z}_{h},s_{h}})} ).$$
By the standard theory of finite element method, we have the following results.
\begin{lemma}\label{lemma 4.2} With the stable condition \eqref{4.2},\\
1. $S\hspace{-0.07cm}_{V_h}$ is a well-defined idempotent operator from $V$ onto $V_h$;\\
2. The approximation holds:
$${\| {{S\hspace{-0.07cm}_{V_h}}( {\varphi ,\undertilde{y},r} ) - ( {\varphi ,\undertilde{y},r} )} \|_V} \le C\mathop {\inf }\limits_{( {{\psi _h},{\undertilde{z}_h},{s_h}} ) \in {V_h}} {\| {( {\varphi ,\undertilde{y},r} ) - ( {\psi_{h} ,\undertilde{z}_{h},s_{h}} )} \|_{{V}}};$$
3. If ${\| {{S\hspace{-0.07cm}_{V_h}}( {\varphi ,\undertilde{y},r} ) - ( {\varphi ,\undertilde{y},r} )} \|_V} \to 0$ as $h \to 0$ for any $( {\varphi ,\undertilde{y},r} ) \in V$, then $\| {{T_{V_h}} - T_{V}} \| \to 0$ as $h \to 0;$
4. The operator $T_{V_h}$ is well-defined and compact on $V_h\subset V$.
\end{lemma}

\begin{lemma}  Let $\{T_{V_{h}}\}$ be a family of compact operators on $V$, such that $\| {{T_{V_{h}}} - T_{V}} \| _{V \to V}\to 0$ as $h \to 0$. For any $i\in \mathbb{N}^{+}$, it follows that the eigenvalues of  $T_{V_{h}}$ can be ordered as
$$\mu _{1,h} \eqslantgtr \mu _{2,h} \eqslantgtr  \cdots  \eqslantgtr 0 \ \ {\rm{with}}\ \ \mathop {\lim }\limits_{h \to 0} \mu _{i,h} = {\mu _i},
$$
and the eigenvalues of \eqref{3.1} can be ordered as
$$0 \eqslantless \lambda _{1,h} \eqslantless \lambda _{2,h} \eqslantless  \cdots \ \ {\rm{with}}\  \ \mathop {\lim }\limits_{h \to 0} \lambda _{i,h} = {\lambda _i}.
$$
Moreover, if $\mu _{i,h}\neq 0$, it holds that $\lambda _{i,h} \mu _{i,h}=1$.
\end{lemma}

Let ${\lambda _1} \eqslantless {\lambda _2} \eqslantless   \cdots$ be the eigenvalues of \eqref{3.4} and ${\lambda _{1,h}} \eqslantless {\lambda _{2,h}} \eqslantless  \cdots $
be the eigenvalues of \eqref{3.1}, respectively. Denote by $N(\lambda_{i})$ and $N_{h}(\lambda_{i,h})$ the eigenspaces of \eqref{3.4} and \eqref{3.1} associated with $\lambda_{i}$ and $\lambda_{i,h}$, respectively. The gap between $N({\lambda _i})$ and ${N_h}(\lambda _{i,h})$ is defined as
$$\hat \delta (N({\lambda _i}),{N_h}(\lambda _{i,h})) := \max (\delta (N({\lambda _i}),{N_h}(\lambda _{i,h})),\delta ({N_h}(\lambda _{i,h}),N({\lambda _i}))),
$$
with $\delta (N({\lambda _i}),{N_h}(\lambda _{i,h})):= \mathop {\sup }\limits_{\zeta \in N({\lambda _i}), \| \zeta \| = 1} {\rm{dist}}(\zeta,{N_h}(\lambda _{i,h}))$. An abstract estimation holds as below.

\begin{lemma} For any $i\in \mathbb{N}^{+}$, there exists a constant $C$ independent of $h$, such that for a small enough $h$,
$$\hat \delta (N({\lambda _i}),{N_h}(\lambda _{i,h})) \le C\left\| {(I_{V} - {S\hspace{-0.07cm}_{V_h}}){|_{N({\lambda _i})}}} \right\|,
$$
where $I_{V}$ is the identity operator on $V$.
\end{lemma}

Let ${\undertilde{H}^1}(\Omega ): = (H^{1}(\Omega))^{d} $, the discretization \eqref{3.1} can be carried out with $V_h$, and a first order accuracy for
eigenfunctions and second order accuracy for eigenvalues can be expected.
\begin{theorem}\label{Thm:order}
For any $i\in \mathbb{N}^{+}$, if $N(\lambda_{i})\subset ({H^2}(\Omega ) \times {\undertilde{H}^1}(\Omega ) \times {H^2}(\Omega ) )\cap V$, then
$$\hat \delta (N({\lambda _i}),{N_h}(\lambda _{i,h})) \le C(N({\lambda _i})){h}\ \  {\rm{and}}\ \  |\lambda_{i} - \lambda _{i,h}|\le Ch^2.
$$
\end{theorem}
\begin{proof} We only have to note that $N({\lambda _i})$ is of finite (fixed) dimension, and any two norms on
that are equivalent. The result follows from the standard theory of finite element methods.
\end{proof}

Let $n_{e_{0}}$, $n_{t}$ and $n_{e}$ be the dimension of $\mathcal{L}_{h0}$, $\undertilde{\mathcal{C}}_{h}$ and $\mathcal{L}_{h}$, respectively, and $\{ {{\phi _1},{\phi _2}, \ldots ,{\phi _{{n_{e_{0}}}}}} \}$, $\{ {{\undertilde{\chi} _1},{\undertilde{\chi} _2}, \ldots ,{\undertilde{\chi} _{{n_t}}}} \}$ and $\{ {{\phi _1},{\phi _2}, \ldots ,{\phi _{{n_{e}}}}} \}$ be the finite element basis of $\mathcal{L}_{h0}$, $\undertilde{\mathcal{C}}_{h}$ and $\mathcal{L}_{h}$, respectively. Then, we have
\begin{equation}\label{Eq3.3}
{\varphi _h} = \sum\limits_{i = 1}^{n_{e_{0}}} {{\omega _i}{\phi _i}}, \  \undertilde{y}_{h} = \sum\limits_{i = 1}^{n_t} {{\eta _i}{\undertilde{\chi} _i}} \ \ {\rm{and}}\ \   r_{h} = \sum\limits_{i = 1}^{n_{e}} {{\gamma _i}{\phi _i}},
\end{equation}
and $\omega  = [ {{\omega _1},{\omega _2}, \ldots ,{\omega _{n_{e_{0}}}}} ]^\top$, $\eta  = [ {{\eta _1},{\eta _2}, \ldots ,{\eta _{n_t}}} ]^\top$ and $\gamma  = [ {{\gamma _1},{\gamma _2}, \ldots ,{\gamma _{n_{e}}}} ]^\top$. We further specify the stiffness, mass and convection matrices in Table \ref{T0}. Moreover, in the discrete setting zero mean value on $\varphi_{h}$ and $r_{h}$ means that
\begin{equation}\label{Eq3.4}
(\varphi_{h},1) = \sum\limits_{i = 1}^{n_{e_{0}}} {{\omega _i}({\phi _i},1)} = \alpha^\top \omega=0\ \
{\rm{and}}\ \
(r_{h},1) = \sum\limits_{i = 1}^{n_{e}} {{\gamma _i}({\phi _i},1)} = \beta^\top \gamma=0,
\end{equation}
where $$\alpha = [(\phi_{1},1),(\phi_{2},1), \ldots , (\phi_{{n_{e_{0}}}},1)]^\top \
{\rm{and}}\ \ \beta = [(\phi_{1},1),(\phi_{2},1), \ldots , (\phi_{{n_{e}}},1)]^\top.$$

\begin{table}
\begin{center}
\begin{tabular}{ccc}
  \hline
  Matrix&Dimension& Definition\\
  \hline
  $K_{P}$&$n_{e_{0}}\times n_{e_{0}}$&\makecell[l]{$(K_{P})_{ij} = ( {{P}\nabla \phi_{j} ,\nabla \phi_{i} })$}\\
  $F_{P}$&$n_{{e_{0}}}\times n_{t}$&\makecell[l]{$(F_{P})_{ij} = ( {P \undertilde{\chi}{_{j}},\nabla \phi_{i} } )$}\\
  $M_{P}$&$n_{t}\times n_{t}$&\makecell[l]{$(M_{P})_{ij} = ( {( {I + {P}} )\undertilde{\chi}_{j},\undertilde{\chi}_{i}} )$}\\
  $G$&$n_{t}\times n_{e}$&\makecell[l]{$(G)_{ij} = ( {\nabla \phi_{j},\undertilde{\chi}_{i}} )$}\\
  $X$&$n_{e_{0}}\times n_{e_{0}}$&\makecell[l]{$(X)_{ij} =  ( {\phi_{j} ,\phi_{i} } )$}\\
  $Y$&$n_{e_{0}}\times n_{e}$&\makecell[l]{$(Y)_{ij} =  ( {\phi_{j},\phi_{i} } )$}\\
    \hline
\end{tabular}
    \caption{Stiffness, mass and convection matrices.}\label{T0}
  \end{center}
  \vspace{-1em}
\end{table}

The constraint \eqref{Eq3.4} is taken into consideration by introducing Lagrangian multipliers $\sigma$ and $\varsigma$. Substituting \eqref{Eq3.3} and \eqref{Eq3.4} multiplied by $\sigma$ and $\varsigma$ into \eqref{3.1}, the discretization gives rise to a GEP\\
\begin{equation}\label{Eq3.5}
\mathcal{K}z = \lambda \mathcal{M}z,
\end{equation}
where $$\mathcal{K} = \left[ {\begin{array}{*{20}{c}}
{{K_{P}}}&{{F_{P}}}&O&\alpha &O\\
{F_{P}^\top}&{{M_{P}}}&{ - G}&O&O\\
O&{ - {G^\top}}&O&O&\beta \\
{{\alpha ^\top}}&O&O&O&O\\
O&O&{{\beta ^\top}}&O&O
\end{array}} \right], \mathcal{M} = \left[ {\begin{array}{*{20}{c}}
X&O&Y&O&O\\
O&O&O&O&O\\
{{Y^\top}}&O&O&O&O\\
O&O&O&O&O\\
O&O&O&O&O
\end{array}} \right]\  {\rm{and}}\  z = \left[ {\begin{array}{*{20}{c}}
\omega \\
\eta \\
\gamma \\
\sigma\\
\varsigma
\end{array}} \right].
$$
In this paper, only a few smallest real eigenvalues of GEP \eqref{Eq3.5} are desired as the approximate transmission eigenvalues.

\section{Preprocessing GEP \eqref{Eq3.5}}\label{sec4}
Let ${\hat n_t} = n_{t}/d$, by definition, $M_{P}=(I+P) \otimes I_{\hat n_t}$ is a block diagonal matrix with $M_{P}^{-1}=(I+P)^{-1} \otimes I_{\hat n_t}$. Define the invertible matrix
\begin{equation*}
{\mathcal{W}} = \left[ {\begin{array}{*{20}{c}}
I&{ - {F_p}M_p^{ - 1}}&O&O&O\\
O&{{G^\top}M_P^{ - 1}}&I&O&O\\
O&O&O&I&O\\
O&O&O&O&I\\
O&I&O&O&O
\end{array}} \right],
\end{equation*}
then the GEP \eqref{Eq3.5} is equivalent to the GEP
\begin{equation}\label{Eq3.6}
({\mathcal{W}}\mathcal{K}{\mathcal{W}^\top})(\mathcal{W}^{ - \top}z) = \lambda({\mathcal{W}}\mathcal{M}{\mathcal{W}^\top})(\mathcal{W}^{ - \top}z),
\end{equation}
where
\begin{equation*}
{\mathcal{W}}\mathcal{K}{\mathcal{W}^\top}=\left[ {\begin{array}{*{20}{cccc:c}}
\hat{K}&\hat{F}&\alpha &O&O\\
\hat{F}^{\top}&\hat{G}&O&\beta &O\\
{{\alpha ^\top}}&O&O&O&O\\
O&{{\beta ^\top}}&O&O&O\\
\hdashline
O&O&O&O&{{M_P}}
\end{array}} \right] := \hat{\mathcal{K}} \oplus M_{P}
\end{equation*}
with
\begin{equation}\label{EQhat_k}
\hat{K} = {{K_{P}} - {F_{P}}{M_{P}}^{ - 1}F_{P}^\top}, \ \hat{F }= {F_{P}{M_{P}^{ - 1}}G}\ \ {\rm{and}}\ \ \hat{G} = { - {G^\top}{M_{P}^{ - 1}}G},
\end{equation}
and
\begin{equation*}
{\mathcal{W}}\mathcal{M}{\mathcal{W}^\top} = \left[ {\begin{array}{*{20}{cccc:c}}
X&Y&O&O&O\\
{{Y^\top}}&O&O&O&O\\
O&O&O&O&O\\
O&O&O&O&O\\
\hdashline
O&O&O&O&O
\end{array}} \right]:=\hat{\mathcal{M}}\oplus O.
\end{equation*}
Now we consider the following GEP
\begin{equation}\label{Eq3.7}
\hat{\mathcal{K}}\hat{z}=\lambda \hat{\mathcal{M}}\hat{z},
\end{equation}
where $  \hat{z} = [ \omega, \gamma, \sigma,\varsigma]^\top$. According to the analysis above, the following theorem can be derived immediately.
\begin{theorem}\label{Th3.8}
The GEPs \eqref{Eq3.5} and \eqref{Eq3.7} have the same non-infinite eigenvalues, and
$${\rm{Spec}}(\hat{\mathcal{K}},\hat{\mathcal{M}}) = {\rm{Spec}}(\mathcal{K},\mathcal{M}) \backslash \{\infty\}_{k=1}^{n_t},$$
where ${\rm{Spec}}(\hat{\mathcal{K}},\hat{\mathcal{M}})$ denotes the set of eigenvalues of the linear pencil $\hat{\mathcal{K}}-\lambda\hat{\mathcal{M}}$.
\end{theorem}
\begin{proof}
The GEP \eqref{Eq3.5} has the same  eigenvalues as GEP \eqref{Eq3.6}, and according to the block structure of the GEP  \eqref{Eq3.6} and the definition of the GEP \eqref{Eq3.7}, this theorem holds obviously.
\end{proof}

Note that the deflation shown above drastically reduces the size of the GEP under consideration. Specifically, one can see that in 2D case the ratio of the number of triangles to that of nodes is about 2, while in 3D case the ratio of the number of tetrahedrons to that of nodes is about 6. In other words, $n_t/n_e \sim 4$ in 2D case and $n_t/n_e \sim$ 18 in 3D case. Therefore, a drastic reduction of the matrix size is expected.

Moreover, unlike several Schur complement based techniques where additional linear systems should be solved, here, the inverse of $M_P$ is rather trivial due to the block diagonal structure of $M_P$. Even more interesting, we show instantly that $\hat K$, $\hat F$ and $\hat G$ in \eqref{EQhat_k} can be directly assembled without carrying out any matrix multiplications in their formal definitions.


\begin{theorem}\label{assemble hatK} Entries of the matrices $\hat K$, $\hat F$ and $\hat G$ in \eqref{EQhat_k} are explicitly given as follows:
$$
\begin{aligned}{(\hat K)_{ij}}&  = ({P}\nabla {\phi _j},\nabla {\phi _i}) - ({P}{(I + {P})^{ - 1}}{P}\nabla {\phi _j},\nabla {\phi _i})= ({A{\nabla\phi _j},\nabla {\phi _i}}),\  i,j = 1,2,\ldots,n_{e_{0}},\\
{(\hat F)_{ij}}& = ({P}{(I + {P})^{ - 1}}\nabla {\phi _j},\nabla {\phi _i})= ({A}\nabla {\phi _j},\nabla {\phi _i}),\ i = 1,2,\ldots,n_{e_{0}},\ j = 1,2,\ldots,n_{e},\\
{(\hat G)_{ij}}& = -({(I + {P})^{ - 1}}\nabla {\phi _j},\nabla {\phi _i})= ({(A - I})\nabla {\phi _j},\nabla {\phi _i}),\ i,j = 1,2,\ldots,n_{e}.
\end{aligned}
$$
\end{theorem}

By Theorem \ref{assemble hatK}, the sparsity of $\hat K$, $\hat F$ and $\hat G$ is very similar to that of $K_p$, which implies GEP \eqref{Eq3.7} is still very sparse besides the significant reduction of the matrix size. This is extremely helpful to practical calculations. In the numerical examples in 3D case presented in Section \ref{sec5}, we note that solving GEP \eqref{Eq3.7} takes much less memory and time than GEP \eqref{Eq3.5} to calculate the smallest few real transmission eigenvalues with the same mesh.


\section{ Numerical experiments on the mixed element discretization scheme}\label{sec5}
In this section, we present some numerical results of transmission eigenvalue problems \eqref{3.4} and \eqref{3.40} on convex and nonconvex domains. All calculations are performed using MATLAB 2020a installed on a workstation with memory 768G and Intel(R) Xeon(R) Gold 6244 CPU @ 3.60GHz.



Unless otherwise specified, the mesh size in the numerical examples below is set to $h_l = 2^{-3-l}$, $l = 1,2,3,4,5,6$ in 2D and $l = 0,1,2,3$ in 3D. Given a mesh size $h_l$, we obtain a sequence of eigenvalues ${\lambda_{i,h_l}}$, $i\in \mathbb{N}^+$. Dof1 and Dof2 represent the degrees of freedom of the GEP \eqref{Eq3.5} and \eqref{Eq3.7}, respectively. The rate of convergence of the transmission eigenvalue is computed by
$$ {\log _2}\Bigg(\Big|\frac{{{\lambda _{{i,h_{l+1}}}} - {\lambda _{{i,h_l}}}}}{{{\lambda_{{i,h_{l+2}}}} - {\lambda _{{i,h_{l + 1}}}}}}\Big|\Bigg),l = 1,2,3,4 {\rm{\ in \ 2D\ and }}\ l = 0,1{\rm{\ in \ 3D. }}
$$

\subsection{ 2D examples}

Here, we are concerned with four domains in 2D \cite{Cakoni2010,Ji20133} illustrated in Figure \ref{fig: figure1}, of which two are the convex domains and the rest are the nonconvex domains. Specifically, these domains are
$\Omega_1 = \{(x,y)|0\leq x^2 + y^2 \leq 1/4\}$, $\Omega_2 =[0,1]^2$, $\Omega_3= [-0.5,0.5] \times [-0.5, 0.5]\backslash (0, 0.5] \times [-0.5, 0)$ and $ \Omega_4 =\{(x,y)|1/16\leq x^2 + y^2 \leq 1/4\}$. $A$ in \eqref{2.4} can be one of the following 2-by-2 matrices
$$
{A_1} = \frac{1}{4}I_2, \ {A_2} = \left[ {\begin{array}{*{20}{c}}
{1/2}&0\\
0&{1/8}
\end{array}} \right]
, \ {A_3} = \left[ {\begin{array}{*{20}{c}}
{1/6}&0\\
0&{1/8}
\end{array}} \right]
,\ {A_4} = \left[ {\begin{array}{*{20}{c}}
0.1372&0.0189\\
0.0189&0.1545
\end{array}} \right].$$
\begin{figure*}[!htb]
    \centering
    \subfigure[ The initial mesh of $\Omega_{1}$.]   {\includegraphics[width=0.4\hsize, height=0.3\hsize]{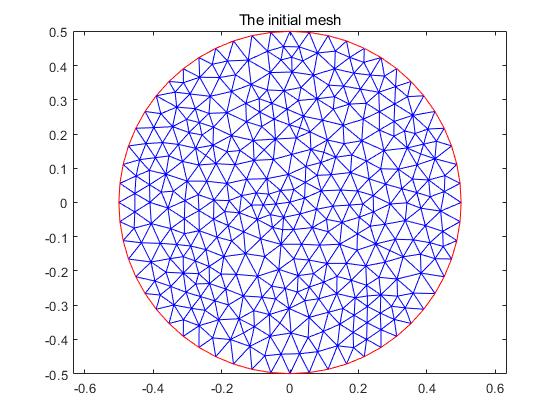}\label{1a}}
    \subfigure[ The initial mesh of $\Omega_{2}$.]{\includegraphics[width=0.4\hsize, height=0.3\hsize]{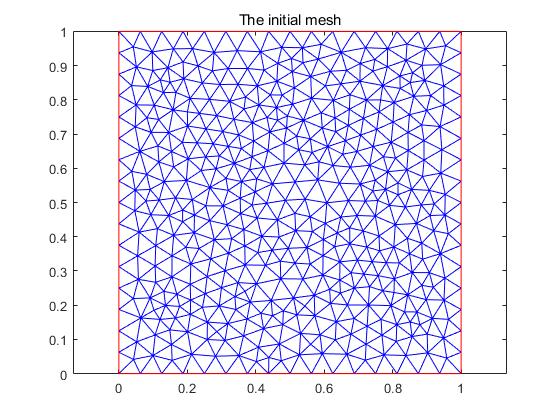}\label{1b}}
    \subfigure[The initial mesh of $\Omega_{3}$.]    {\includegraphics[width=0.4\hsize, height=0.3\hsize]{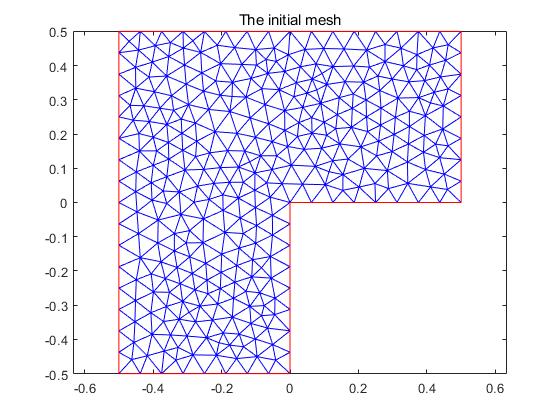}\label{1c}}
    \subfigure[The initial mesh of $\Omega_{4}$.]         {\includegraphics[width=0.4\hsize, height=0.3\hsize]{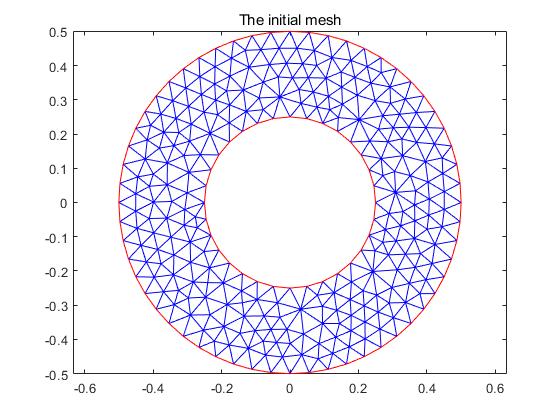}\label{1d}}
    \caption{ The initial meshes of domains for 2D examples.}
    \label{fig: figure1}
\end{figure*}

\noindent We first verify the proposed scheme by comparing the smallest real transmission eigenvalue of

{\bf{Example 1 }} $\Omega_1$ with the anisotropic case $A_1$\\
calculated by our MATLAB implementation with the known exact result, and then we investigate the following examples:

{\bf{Example 2 }} \ $\Omega_2$ with the anisotropic cases $A_i, i = 1, 2, 3, 4$, respectively.

{\bf{Example 3 }} \ $\Omega_3$ with the anisotropic cases $A_i, i = 1, 2, 3, 4$, respectively.

{\bf{Example 4 }} \ $\Omega_4$ with the anisotropic cases $A_i, i = 1, 2, 3, 4$, respectively.

We compute the first six real transmission eigenvalues $(k_{i,h_{l}} = \sqrt{\lambda_{i,h_{l}}}) $ of
Example 1 on the mesh level $h_{1}$ to $h_{6}$ and the smallest real transmission eigenvalues of
Example 2, 3 and 4 on the mesh level $h_{6}$. The results are listed in Table \ref{T1} and Table \ref{T2}. Dof1/Dof2 is about 3 as expected. In particular, the first real transmission eigenvalue of Example 1 calculated by our method is 5.8053, which is consistent with the exact transmission eigenvalue 5.8 in \cite{Cakoni2010,Ji20133}. And in Example 1, it takes about 4.3 GB memory with 101 seconds to solve the GEP \eqref{Eq3.7} on the mesh level $h_6$. Moreover, as the mesh is refined, the calculated real eigenvalues sequence monotonically decreases to the exact results. We show the convergence rate of eigenvalues in Section \ref{sec4.3}.

\begin{table}
\begin{center}
\begin{tabular}{ccccccccc}
  \hline
   Mesh &Dof1&Dof2&$k_{1,h_{l}}$& $k_{2,h_{l}}$& $k_{3,h_{l}}$& $k_{4,h_{l}}$& $k_{5,h_{l}}$&$k_{6,h_{l}}$\\
  \hline

   $l = 1$&2116&708& 5.8978& 6.9404& 6.9541&  7.8375& 7.8555& 7.8845\\

  $ l = 2$&8056&2688& 5.8301& 6.8386& 6.8392& 7.6392& 7.6405&  7.6747\\

  $l = 3$&31792&10600& 5.8117& 6.8108& 6.8109& 7.5850& 7.5851& 7.6239\\

  $l = 4$&132796&44268 & 5.8067&  6.8031& 6.8031& 7.5705&  7.5705& 7.6107\\

  $l = 5$&536656&178888& 5.8056& 6.8013& 6.8013& 7.5671& 7.5671&  7.6076\\

  $ l =6$&2107438&702482&  5.8053& 6.8009& 6.8009&  7.5663&  7.5663& 7.6069\\
    \hline
\end{tabular}
    \caption{ The first six real transmission eigenvalues of Example 1.}\label{T1}
  \end{center}
\end{table}

\begin{table}
\begin{center}
\begin{tabular}{ccccccc}
  \hline
 $\Omega$ &Dof1 &Dof2&$A_1$&$A_2$&$A_3$&$A_4$\\
\hline
  $\Omega_2$&2722144&907384&5.2987&4.3867&3.5816& 3.6105 \\

  $\Omega_3$&2041984&680664&6.7284&5.9350& 4.3026& 4.3514\\

  $\Omega_4$&1605368&535124&11.3526&11.6678& 7.1936&7.1943\\
    \hline
\end{tabular}
    \caption{ The smallest real transmission eigenvalues of Example 2, 3 and 4.}\label{T2}
  \end{center}
    \vspace{-1em}
\end{table}

\subsection{ 3D examples}\label{sec4.2}
Here, we are concerned with four domains in 3D illustrated in Figure \ref{Fig2}, of which two are the convex domains and the rest are the nonconvex domains. Specifically, these domains are
$\Omega_{5}=B(0,1)$, $\Omega_{6}=[0,1]^{3}$,
$\Omega_{7}=[0,1]^{3}\backslash (0.25,0.75)^{3}$ and $\Omega_{8}= [0,1]^{3} \backslash \{(x,y,z)|(x-1/2)^{2} + (y-1/2)^{2} < 1/16,\ 0 \leq z\leq 1\}$.
 $A$ in \eqref{2.4} can be one of the following 3-by-3 matrices
$${A_5} = 11I_3, \  {A_6} = \frac{1}{4}I_3,\ {A_7} = \left[ {\begin{array}{*{20}{c}}
{1/4}&0&0\\
0&{1/2}&0\\
0&0&{1/8}
\end{array}} \right],\  {A_8} = \left[ {\begin{array}{*{20}{c}}
  1/4&-1/8& 0\\
-1/8&3/8&-1/8\\
0&-1/8&1/4
\end{array}} \right].
$$

\begin{figure*}[!htb]
    \centering
    \subfigure[$\Omega_{5}$: The unit ball domain.]       {\includegraphics[width=0.48\hsize, height=0.36\hsize]{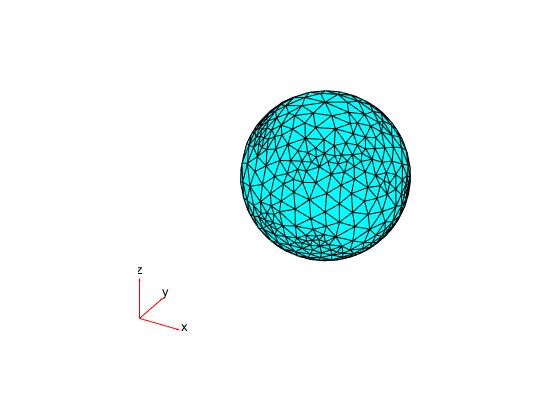}\label{2a}}
    \subfigure[$\Omega_{6}$: The unit cube domain.]       {\includegraphics[width=0.48\hsize, height=0.36\hsize]{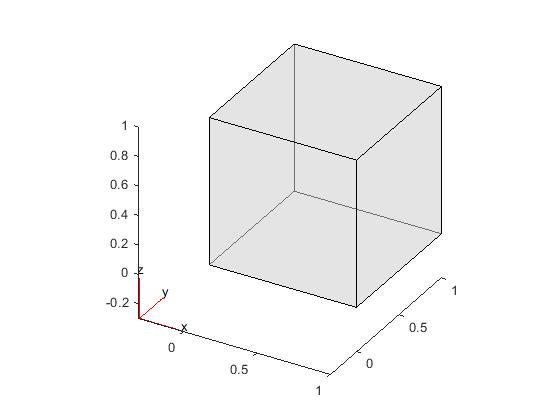}\label{2b}}
    \subfigure[$\Omega_{7}$: The unit cube domain with a inside cube cavity.]   {\includegraphics[width=0.48\hsize, height=0.36\hsize]{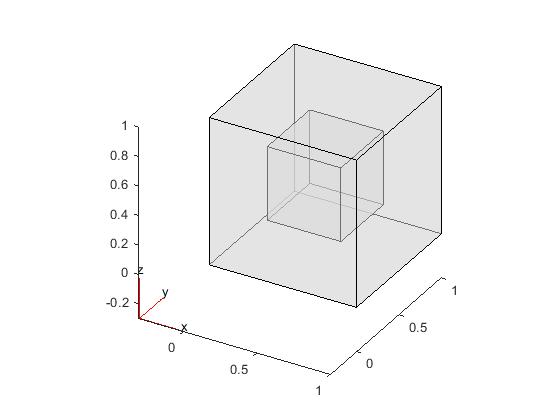}\label{2c}}
    \subfigure[$\Omega_{8}$: The unit cube domain with a cylinder cavity.]{\includegraphics[width=0.48\hsize, height=0.36\hsize]{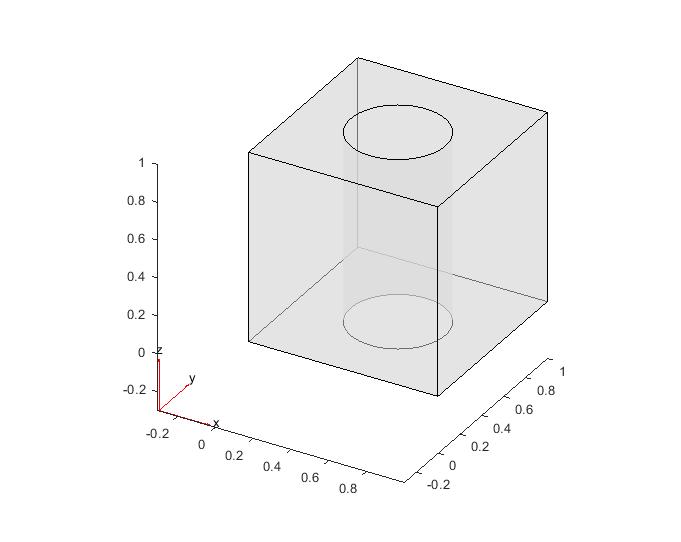}\label{2d}}
    \caption{Domains of 3D examples.}
    \label{Fig2}
\end{figure*}

\noindent We first verify the proposed scheme by comparing the first four real transmission eigenvalues of

{\bf{Example 5 }} $\Omega_5$ with the anisotropic case $A_5$\\
calculated by our MATLAB implementation with the known exact results in \cite{Lechleiter2015}, and then we investigate the following examples:

{\bf{Example 6 }} \ $\Omega_6$ with the anisotropic cases $A_i, i = 6, 7, 8$, respectively.

{\bf{Example 7 }} \ $\Omega_7$ with the anisotropic cases $A_i, i = 6, 7, 8$, respectively.

{\bf{Example 8 }} \ $\Omega_8$ with the anisotropic cases $A_i, i =  6, 7, 8$, respectively.

We compute the first six real eigenvalues $(k_{i,h_{l}} = \sqrt{\lambda_{i,h_{l}}}) $ of Example 5 on the mesh level $h_{0}$ to $h_{3}$ and Example 6, 7 and 8 on the
mesh level $h_3$. The results are listed in Table \ref{T5} and Table \ref{T6}. Dof1/Dof2 is about 10 as expected. It is worth noting that the first four transmission eigenvalues of Example 5 calculated by our method are 5.2052, 5.8882, 6.1057 and 7.2469 indicated in Table \ref{T5}, which are consistent with the exact eigenvalues 5.204, 5.886, 6.104 and 7.244 in \cite{Lechleiter2015}. Moreover, as the mesh is refined, the calculated real eigenvalues sequence monotonically decreases to the exact results.

In Example 5, it takes about 2.6, 9.3 and 178 GB memory with 5.7, 326 and 16943 seconds to calculate the smallest real transmission eigenvalue by solving the GEP \eqref{Eq3.5} on the mesh level $h_0$, $h_1$ and $h_2$, respectively, while correspondingly solving the GEP \eqref{Eq3.7} only takes 2.3, 5.5 and 46.4 GB memory with 4.3, 168 and 5845 seconds. That is, both the computational cost and time are significantly reduced by the preprocessing in section \ref{sec4}.

\begin{table}
\begin{center}
\begin{tabular}{ccccccccc}
  \hline
    Mesh&Dof1&Dof2&$k_{1,h_{l}}$& $k_{2,h_{l}}$& $k_{3,h_{l}}$& $k_{4,h_{l}}$& $k_{5,h_{l}}$&$k_{6,h_{l}}$\\
  \hline

  $l=0$ &77093&7796 & 5.2964&  5.9969&  6.2199&  7.4322& 8.6910&  9.7576\\

  $l=1$ &532126&53500&  5.2281&  5.9160&  6.1348&  7.2947&  8.4790&  9.4481\\

  $l=2$ &3805418&382271 & 5.2097&   5.8936&  6.1113& 7.2562&  8.4193&  9.3611\\

  $l=3$ &28970976&2907867&  5.2052&  5.8882&  6.1057 &7.2469 &  8.4048 & 9.3402\\

    \hline
\end{tabular}
    \caption{ The first six real transmission eigenvalues of Example 5.}\label{T5}
  \end{center}
\end{table}







\begin{table}
\begin{center}
\begin{tabular}{cccccccccc}
  \hline
  $\Omega$&$A$&Dof1&Dof2&$k_{1,h_{3}}$& $k_{2,h_{3}}$& $k_{3,h_{3}}$& $k_{4,h_{3}}$& $k_{5,h_{3}}$&$k_{6,h_{3}}$\\
  \hline
  &$A_{6}$&&& 5.2602& 5.9167&5.9168& 5.9168& 6.3506& 6.3506  \\

 $\Omega_6$& $A_{7}$&7201086&722772&4.7106& 5.0600& 5.5420& 5.7049& 6.2222& 6.3625  \\

 & $A_{8}$&&& 4.8690& 4.9720& 5.7786& 5.7893& 5.9318& 6.2632\\
    \hline

& $A_{6}$&&& 10.1747& 10.3596& 10.3609& 10.3620& 10.4774&  10.4789 \\

  $\Omega_7$ &$A_{7}$&6294048&630462& 8.7369& 8.8006& 8.8962& 8.9308&9.5239& 9.5707  \\

  &$A_{8}$&&& 9.2961& 9.3513& 9.4267& 9.4805& 9.6826& 9.9264\\
    \hline
    &$A_{6}$& &&9.4155& 9.5137& 9.5139& 9.5182& 9.5311&  9.5965  \\

  $\Omega_8$ & $A_{7}$&5711887&572794& 9.0059& 9.0089& 9.1381&9.1395& 9.2460& 9.2496  \\

  &$A_{8}$&&& 8.6118& 8.6170& 8.7366&8.7459&9.5856&  9.5943\\
    \hline
\end{tabular}
    \caption{ The first six real transmission eigenvalues of Example 6, 7 and 8.}\label{T6}
  \end{center}
    \vspace{-0.5em}
\end{table}









\subsection{ Convergence rate of transmission eigenvalues }\label{sec4.3}
 Figure \ref{Fig3} demonstrates that the discretization \eqref{3.1} can be carried out with $V_{h}$ for the 2D and 3D cases, and a second order accuracy for eigenvalues can be obtained both on convex and nonconvex domains. Therefore, the optimal convergence rate for eigenvalues can be obtained by this discrete mixed element scheme, which is consistent with the Theorem \ref{Thm:order}.
\begin{figure*}[!htb]
    \centering
    \subfigure[$\Omega_{1}$ with the anisotropic case $A_{1}$.]       {\includegraphics[width=0.40\hsize, height=0.24\hsize]{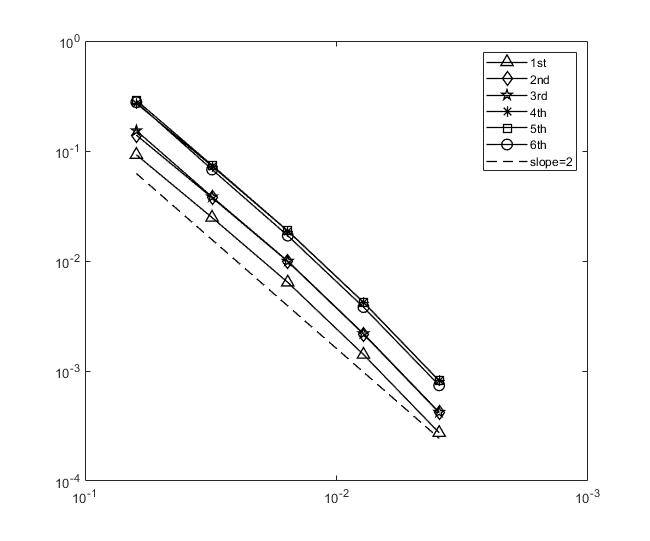}\label{3a}}
    \subfigure[$\Omega_{2}$ with the anisotropic case $A_{2}$.]      {\includegraphics[width=0.40\hsize, height=0.24\hsize]{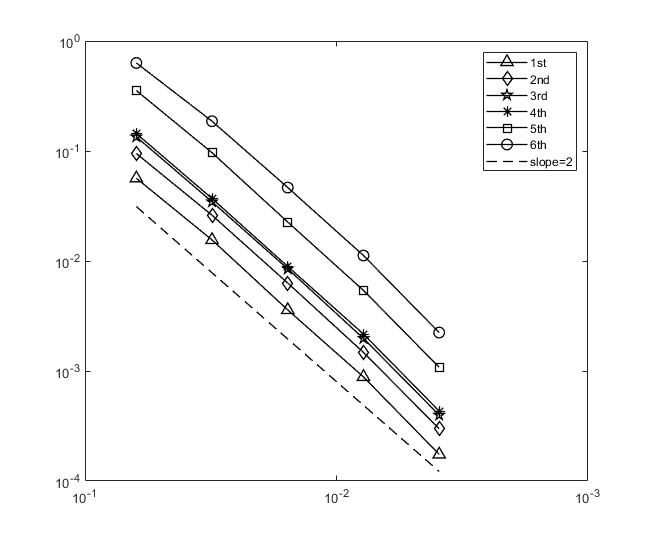}\label{3b}}
    \subfigure[$\Omega_{3}$ with the anisotropic case $A_{3}$.]       {\includegraphics[width=0.40\hsize, height=0.24\hsize]{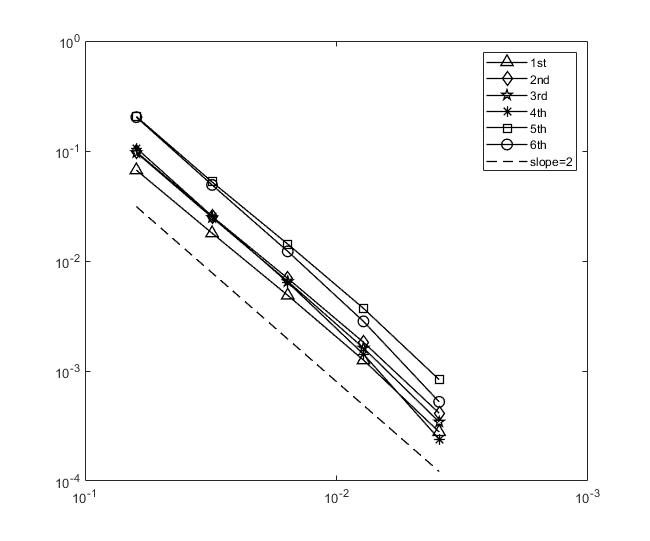}\label{3c}}
    \subfigure[$\Omega_{4}$ with the anisotropic case $A_{4}$.]       {\includegraphics[width=0.40\hsize, height=0.24\hsize]{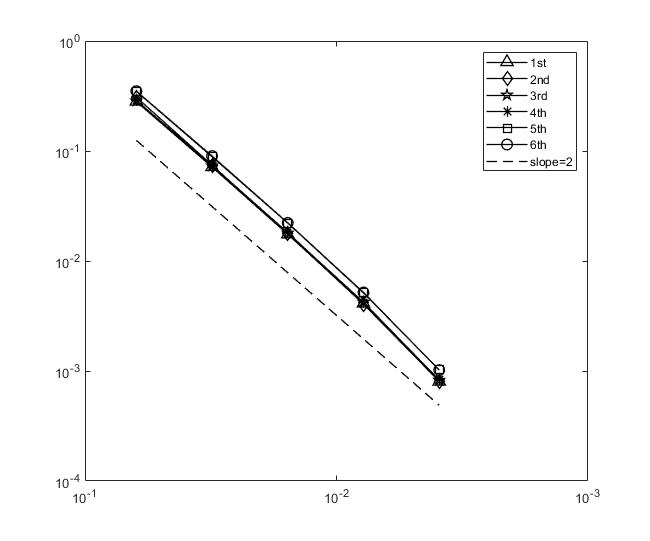}\label{3d}}
    \subfigure[$\Omega_{5}$ with the anisotropic case $A_{5}$.]       {\includegraphics[width=0.40\hsize, height=0.24\hsize]{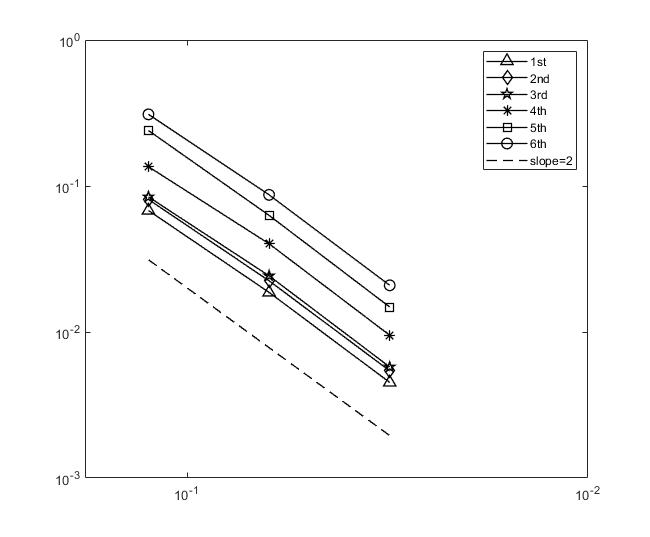}\label{5a}}
    \subfigure[$\Omega_{6}$ with the anisotropic case $A_{6}$.]      {\includegraphics[width=0.40\hsize, height=0.24\hsize]{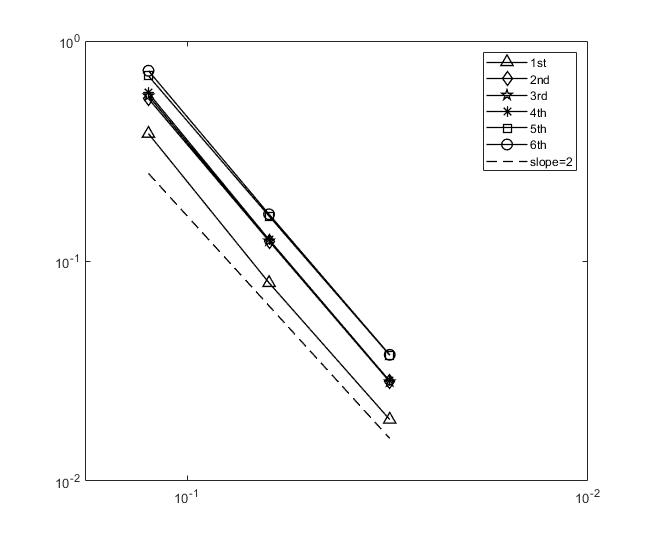}\label{5b}}
    \subfigure[$\Omega_{7}$ with the anisotropic case $A_{7}$.]       {\includegraphics[width=0.40\hsize, height=0.24\hsize]{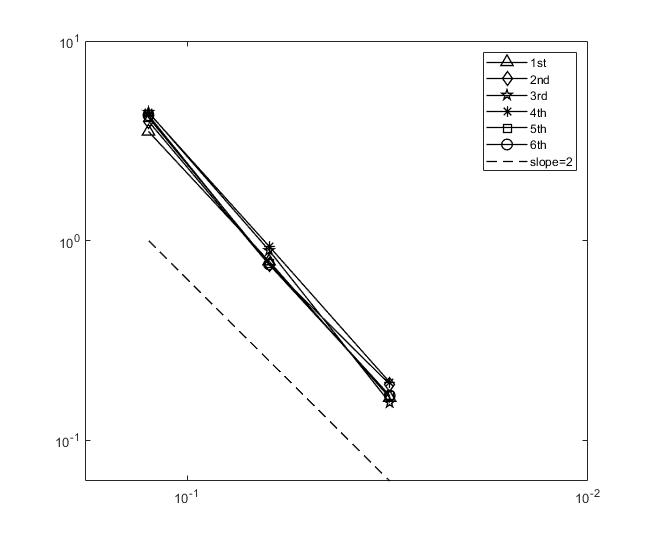}\label{5c}}
    \subfigure[$\Omega_{8}$ with the anisotropic case $A_{8}$.]       {\includegraphics[width=0.40\hsize, height=0.24\hsize]{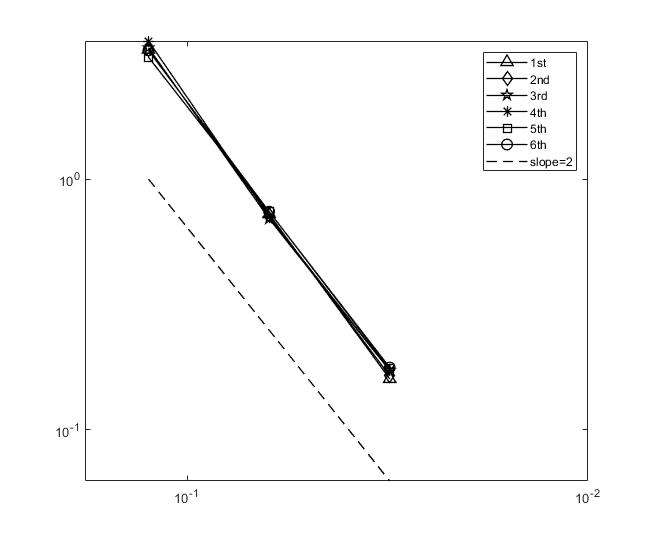}\label{5d}}
    \caption{ The convergence rates for 2 D and 3D examples. X-axis means the size of mesh and Y-axis means $|\lambda_{i,{h}_{l+1}}-\lambda_{i,{h}_{l}}|$.}
    \label{Fig3}
\end{figure*}

\section{ Conclusions and discussions}\label{sec6}
In this paper, we discuss the Helmholtz transmission eigenvalue problem for anisotropic inhomogeneous media with the index of refraction $n(x)\equiv 1$ in 2D and 3D, and its discretization using a mixed finite element scheme. Our key theoretical result is that the proposed mixed formulation is totally equivalent to {\eqref{2.4}} without introducing any spurious eigenvalues, and the proposed discretization scheme is easy to implement with firm theoretical support. In practice, the goal of computing a few smallest real transmission eigenvalues is hindered by the tremendous size of the resulting GEP, in that the $\boldsymbol{LDL}$ factorization cannot fit into the computer memory. We partially resolve this critical issue by the carefully desinged preprocessing shown in Section \ref{sec4}. With the prepocessing, not only the almost all of the huge eigenspace associated with the $\infty$ eigenvalue is  deflated, but also the size of the GEP is drastically reduced without deteriorating the sparsity, hence, the $\boldsymbol{LDL}$ factorization become feasible. Both the computational time and cost are significantly reduced, especially in the 3D case.

Numerical examples on the convex and nonconvex connected domains both in 2D and 3D confirm that the optimal convergence rate of the transmission eigenvalues can be achieved by the proposed scheme, which is consistent with the theoretical predication.

In the future, we will try to generalize this scheme to other types of transmission eigenvalue problems, such as elastic waves and Maxwell transmission eigenvalue problems. Further studies will be carried out so that the memory cost can be controlled as the mesh gets finer in 3D domains.

\section*{ Acknowledgments}
The research of S. Zhang is supported partially by the Strategic Priority Research Program of CAS with Grand No. XDB 41000000; the National Natural Science Foundation of China (NSFC) with Grant No. 11871465. Q. Liu and T. Li are supported partially by the NSFC with Grant No. 11971105 and the
Big Data Computing Center in Southeast University, China.

%
%

%
%


\end{document}